\documentclass[11pt]{article}

\usepackage[T1]{fontenc} 
\usepackage{lmodern}
\usepackage{blindtext}

\usepackage[utf8]{inputenc}
\usepackage{amsfonts} 
\usepackage{dsfont}
\usepackage{amsmath}
\usepackage{amsmath,latexsym}
\usepackage{amsmath,amssymb}
\usepackage{amsmath,amsthm}
\usepackage[english]{babel}
\usepackage{moreverb}
\usepackage[]{graphics}
\usepackage[]{latexsym} 
\usepackage{amscd}
\usepackage{color}
\usepackage{lineno}
\usepackage{ae}
\usepackage{enumerate}
\usepackage{multicol} 

\usepackage[figuresright]{rotating}
\usepackage{bigints}
\graphicspath{{FIGURES/}}
\makeatletter
\newcommand{\xrightharpoonup}[2][]{\ext@arrow 0359\rightharpoonupfill@{#1}{#2}}
\makeatother
\usepackage{hyperref}
\hypersetup{
	colorlinks=true,
	linkcolor=blue,
	filecolor=magenta,      
	urlcolor=cyan,
}

\usepackage{graphics}
\usepackage{graphicx}

\newcommand{\s}{\sigma}  
\newcommand{\ds}{\displaystyle}

\def\R{\mathbb{R}}

\def\P{\mathbb{P}}
\def\E{\mathbb{E}}

\def\N{{\rm I\hspace{-0.50ex}N} } 
\def\N{\mathbb{N}}

\def\notin{\,\hbox{\it / \hskip -10pt $\in$}\,}

\newtheorem{lem}{Lemma}[section]
\newtheorem{thm}[lem]{Theorem}

\newtheorem{dfn}[lem]{Definition}
\newtheorem{prop}[lem]{Proposition}

\title{\bf Locally Optimal Functional Quantization}
\author{ 
{\sc Harald Luschgy}~\thanks{Universit\"at Trier, FB IV-Mathematik, D-54286 Trier, Germany. E-mail: {\tt luschgy@uni-trier.de}} 
\and   
{\sc  Gilles Pag\`es}~\thanks{Sorbonne Universit\'e, Laboratoire de Probabilit\'es, Statistique et Mod\'elisation, case 158, 4, pl. Jussieu, F-75252 Paris Cedex 5, France. E-mail: {\tt  gilles.pages@sorbonne-universite.fr}}~\thanks{This research
benefited from the support of the ``Chaire Risques Financiers'', Fondation du Risque.}  }

\begin{document}
	\maketitle
	
	\begin{abstract} In this note we demonstrate that locally optimal functional quantizers for probability distributions  $P$ on a Banach space lying in the support of $P$ 	behave exactly like globally optimal functional quantizers in terms of stationarity.
	\end{abstract}
	\paragraph{Keywords:} functional quantization ; stationary quantizer ; locally optimal quantizers.
	\bigskip
	\section{Introduction} We will use the notations of~\cite{LuPag23}. Let  $X$ be a Radon random vector on a probability space $(\Omega, {\cal A}, \P)$ taking its values in a Banach space $(E, \|\cdot\|)$ equipped with its Borel $\s$-algebra ${\cal B}or(E)$. This means that $X$ is $\big({\cal A}, {\cal B}or(E)\big)$-measurable and its distribution $P= \P_{_X}$ is a Radon probability measure on ${\cal B}or(E)$ (in the sense that it is inner regular). Notice that stochastic processes in continuous time can often be seen as random vectors having values in a Banach function space.
	
Let 
	\[
	d(x,A):= \inf\{\|x-a\|: \; a\!\in A\}
	\]
denote the distance of $x\!\in E$	to the nonempty subset $A\subset E$. For $n\!\in \N$ and $r\!\in (0, +\infty)$ the $L^r$-quantization problem for $P$ (or $X$) at level $n$ consists in minimizing 
\begin{equation}\label{eq:e_r(Gamma)}
e_r(\Gamma, P) := \big(\E\, d(X, \Gamma)^r \big)^{1/r}
\end{equation}
over all subsets$\Gamma\subset E$, $\Gamma\neq \varnothing$ with ${\rm card}(\Gamma)\le n$. Such a subset is called {\em $n$-quantizer} (or {\em $n$-codebook} or {\em $n$-grid}).

\begin{dfn}$(a)$ The $n$th-level $L^r$-quantization error of $P$ is defined as 
\[
e_{r,n}(P,E) := \inf\big\{ e_r(\Gamma,P): \Gamma\subset E, \; 1\le {\rm card}((\Gamma) \le n \big\}.
\]

\noindent $(b)$ Assume $\E\, \|X\|^r= \int \|x\|^r dP(x) <+\infty$ (so that $e_{r,n}(P,E)<+\infty$ for every $n\ge 1$). A quantizer $\Gamma\subset E$ with ${\rm card}(\Gamma)\le n$ is called $L^r$-optimal $n$-quantizer  for $P$ of 
\[
e_{r}(\Gamma,P)= e_{r,n}(P,E). 
\]
Let ${\cal C}_{r,n}(P,E)$ denote the set of all $L^r$-optimal $n$-quantizers for $P$. 
\end{dfn}	

A Voronoi quantization of $X$ induced by a finite quantizer $\Gamma \subset E$ is simply the projection of $X$ onto $\Gamma$ following the nearest neighbour rule as explained below.

\begin{dfn} Let $\Gamma\subset E$ be a finite quantizer and $a\!\in \Gamma$. Then 
\[
V_a(\Gamma):= \big\{x\!\in E: \|x-a\| = d(x,\Gamma)\big\}
\] 
is called the (closed) Voronoi cell and
\[
 W_a(\Gamma) := \big\{x\!\in E: \|x-a\| <d(x,\Gamma\setminus\{a\})\big\}
\]
the open Voronoi cell of $a$ induced by $\Gamma$. A Borel measurable partition $\{C_a(\Gamma): a\!\in\Gamma\}$ of $E$ is called Voronoi partition induced by $\Gamma$ if
\[
C_a(\Gamma)\subset V_a(\Gamma) \; \mbox{ for every } a\!\in \Gamma.
\]
\end{dfn}		

Since the distance function is continuous, $V_a(\Gamma)$ is in fact a closed set of $E$ and $W_a(\Gamma)$ an open set. It should also be obvious that, for a Voronoi partition $\{C_a(\Gamma): a\!\in \Gamma\}$,
\begin{equation}\label{eq:WsubseC}
W_a(\Gamma)\subset C_a(\Gamma) \mbox{ for every } \; a\!\in \Gamma.
\end{equation}	
	
Given a finite quantizer $\Gamma\subset E$ and a Voronoi partition $\{C_a(\Gamma): a\!\in \Gamma\}$ of $E$ induced by $\Gamma$, define
\[
{\rm Proj}_{\Gamma}: E\to E,\quad {\rm Proj}_{\Gamma}(x) := \sum_{a\in \Gamma}\mbox{\bf 1}_{C_a(\Gamma)}(x).
\]	
Note that ${\rm Proj}_{\Gamma}$ depends on the Voronoi partition and not only on $\Gamma$.	 The random vector
\[
\widehat X =\widehat X^{\Gamma} =  {\rm Proj}_{\Gamma}(X)
\]
is called {\em Voronoi quantization} of $X$ induced by $\Gamma$ or {\em $\Gamma$-quantization} 	 of $X$. Its $L^r$-(mean)error satisfies 
\[
\big( \E\, \|X-\widehat X\|^r\big)^{1/r} = e_r(\Gamma,P)
\]	
and so does not depend on the specific choice of the Voronoi partition induced by $\Gamma$.	Functional quantization can thus be seen as a way to discretize the path space of a stochastic process. So far, we have dealt with (functional) quantizers viewed as a finite subset of $E$. For our present purpose it is convenient to introduce them also a $n$-tuples. The {\em $L^r$-distortion function of level $n$} is defined by 
\begin{equation}\label{eq:r-distor}
G_{r,n}^P : E^n\to \R_+, \; G_{r,n}^P(a) := \E\min_{1\le i\le n} \|X - a_i\|^r.
\end{equation}
Note that $G_{r,n}^P$ is continuous w.r.t. the product topology on $E^n$ (in fact $(G_{r,n}^P)^{1/r}$ is even Lipschitz continuous  with Lipschitz coefficient $1$ w.r.t. the $\max$-norm $\|(x_1,\ldots,x_n)\|_{\ell^{\infty}} :=\max_{1\le i\le n}\|x_i\| $ on $E^n$ if $r \ge 1$), convex for $n=1$ and $r\ge 1$ but typically not convex if $n\ge 2$. We have
\[
e_{r,n}(P,E)^r = \inf_{a\in E^n} G_{r,n}^P(a).
\]
We will introduce the notion of $L^r$-stationarity quantizers as critical points of $G_{r,n}^P$. 

The Banach space	$E$ and its norm $\|\cdot\|$ is said to be smooth if the norm is G\^ateaux-differentiable at any $x\neq 0$, where a function $F:E\to \R$ is {\em G\^ateaux-differentiable} at $x\!\in E$ if there exists $\nabla F(x) \!\in E^*$ (topological dual space of $E$) such that 
\[
\lim_{t\to 0}\frac{F(x+ty)-F(x)}{t} = \langle \nabla F(x),y\rangle
\]   
for every $y\!\in E$. We write $\langle u, y\rangle$ for $u(y)$ for $u\!\in E^*$ and $y\!\in E$.  Then $\nabla \|\cdot\|(x)$ is contained in the unit sphere of $E^*$ when the norm is G\^ateaux-differentiable at $x$. In fact, we only need almost sure smoothness of the norm. (This weaker condition is of interest e.g. for $L^1$-spaces $E$.). 

For convenience, a finite quantizer $\Gamma\subset E$ is called {\em admissible} for $P$ if 
\begin{equation}\label{eq:admissible}
P\left( \bigcup_{a\in \Gamma} W_a(\Gamma)\right) =1.
\end{equation}

An $n$-tuple $(a_1,\ldots,a_n)\!\in E^n$ with $a_i\neq a_j$ is admissible for  $P$ if its associated $n$-quantizer $\{a_1,\ldots,a_n\}$ is. 

\begin{prop} Let $r\ge 1$ and assume $\ds \int_E \|x\|^r dP(x)<+\infty$. Assume that the norm is G\^ateaux-differentiable at $a-x$ $P(dx)$-$a.e.$ on $E\setminus\{a\}$. Let $r>1$. Then the $L^r$-distortion function $G_{r,n}^P$ is G\^ateaux-differentiable at every admissible $n$-tuple $(a_1,\ldots,a_n)$ with a G\^ateaux derivative given by
\[
\nabla G^P_{r,n} (a_1,\ldots, a_n) = r \bigg(\E\, \mbox{\bf 1}_{C_{a_i}(\Gamma)\setminus\{a_i\}}(X) \|X-a_i\|^{r-1}\nabla \|\cdot\|(a_i-X)\bigg)_{1\le i \le n} \!\in (E^*)^n,
\]
where $\{C_{a_i}(\Gamma): 1\le i\le n\}$  denotes any Voronoi partition  induced by $\Gamma=\{a_1,\ldots, a_n\}$.

\smallskip When $r=1$, the above result holds true for admissible $n$-tuples $(a_1,\ldots,a_n)$ satisfying\\ $P(\{a_1,\ldots, a_n\})=0$.
\end{prop}

\noindent {\bf Proof.} See Proposition~1.3.1 in~\cite{LuPag23} .\hfill$\Box$

\bigskip
Note that $(E^n)^*= (E^*)^n$ with duality $\langle u, x\rangle= \sum_{j=1}^n \langle u_j,x_j\rangle$, where $u=(u_1,\ldots,u_n)\!\in (E^*)^n$ and $x=(x_1,\ldots,x_n) \!\in E^n$. Let us point out that the integrals (expectations) occurring in the above formulas for $\nabla G_{r,n}^P$ and in the subsequent definition are to be seen as {\em Gelfand integrals}. Our notion of stationarity avoids admissibility.

\begin{dfn}Let $r\ge 1$ and assume $\ds \int_E \|x\|^rdP(x) <+\infty$. Assume that the norm is G\^ateaux-differentiable at $a-x$ $P(dx)$-$a.e.$ on $E \setminus\{a\}$ for every $a\!\in E$. A quantizer $\Gamma\subset E$ with ${\rm card}(\Gamma) = n$ and $P(\Gamma)=0$ when $r=1$ is called $L^r$stationary $n$-quantizer for $P$  if there exists  a Voronoi partition $\{C_a(\Gamma): a\!\in \Gamma\}$ induced by $\Gamma$ such that $P\big( C_a(\Gamma)\big) >0$ and 
\begin{equation}\label{eq:astionarite0}
 \E\, \mbox{\bf 1}_{C_{a}(\Gamma)\setminus\{a\}}(X) \|X-a\|^{r-1}\nabla \|\cdot\|(a-X)\stackrel{E^*}{=}0
\end{equation}
for every $a\!\in \Gamma$. (Here $x\stackrel{E^*}{=} 0$ means that $x$ is equal to the zero of $E^*$).
\end{dfn}

\section{Local minima of the distortion functions}
Let ${\rm supp}(P)$ denote the support of $P$ which exists in our setting.

Local minima of the $L^r$-distortion function $G_{r,n}^P$ with components lying in ${\rm supp}(P)$ behave exactly like  global minima although $G_{r,n}^P$ is not convex for $n\ge 2$. For global minima see~\cite[Sections 1.2 and 1.3]{LuPag23}. In the subsequent results let $r\!\in(0, +\infty)$ and assume $\ds \int_E \|x\|^rdP(x) <+\infty$.

\begin{thm}\label{thm:1.3.4} 
Let $a=(a_1,\ldots,a_n)$ be a local minimum of $G_{r,n}^P$ with $a_i $ a non-isolated point of ${\rm supp}(P)$  for every $i\!\in \{1,\ldots,n\}$.

\smallskip
\noindent $(a)$ The components $a_i$ of $a$ are pairwise distinct and if $r\ge 1$, then for every $i\!\in \{1, \ldots,n\}$ and every $C \!\in {\cal B}(E)$ with $W_{a_i}\big(\{a_1,\ldots,a_n\}\big)\subset C \subset V_{a_i}\big(\{a_1,\ldots,a_n\}\big)$,
\[
a_i \!\in {\rm argmin}_{y \in E} G_{r,1}^{P(\cdot|C)}(y)
\]
or, equivalently, $\{a_i\} \!\in {\cal C}_{r,1} \big(P(\cdot\,|\,C), E\big)$.

\smallskip
\noindent  $(b)$ Let $r\ge 1$. Assume that the norm is G\^ateaux differentiable at $b-x$ $P(\mathrm{d}x)$-a.e. on $E\setminus \{b\}$ for every element $b$ of $E$ and $P(\{a_i\})=0$ for every $i$ when $r=1$. Then $\Gamma= \{a_1,\ldots,a_n\}$ is an $L^r$-stationary $n$-quantizer for $P$. If $E$ is strictly convex, then $\Gamma$ is admissible for $P$.
\end{thm}

\begin{proof} $(a)$ Let $U\subset E^n$ be an open neighbourhood of $a$ such that $G_{r,n}^P$ has a global minimum at $a$ on $U$. We may assume that $U = \prod_{i=1}^n U_i$ has a product form, where $U_i \subset E$ are open neighbourhoods of $a_i$. Let $N$ denote the number of pairwise distinct components of $a$ so that there exist $i_1=1<i_2<\cdots< i_N  \le n$ such that $a_{i_1}, \ldots,a_{i_N}$ are pairwise distinct. Set $\Gamma= \{a_{i_1},\ldots,a_{i_N}\}$. Assume $N\le n-1$.  Then there exists $i\!\in \{1, \ldots,n\}\!\setminus \!\{i_1,\ldots,i_{_N}\}$.

Choose $c\!\in \big( {\rm supp}(P)\setminus \Gamma\big)\cap U_i \neq \varnothing$ where we use that $a_i$ is not an isolated point of ${\rm supp}(P)$  and set $\Gamma_1 = \Gamma\cup \{c\}$. Let $\{C_b(\Gamma_1): b\!\in \Gamma_1\}$ be a Voronoi partition of $E$ induced by $\Gamma_1$ such that $C_c(\Gamma_1) = W_c(\Gamma_1)$. Now
\[
d(x, \Gamma_1) = \|x-c\| < d(x,\Gamma) \quad\mbox{ on }\quad C_c(\Gamma_1) = W_c(\Gamma_1)
\]
so that, using $P\big(W _c(\Gamma_1)\big)>0$,
\begin{align*}
G_{r,n}^P(a)&= \int d(x, \Gamma)^r \mathrm{d}P(x)\\
& = \sum_{b\in \Gamma} \int_{C_b(\Gamma_1)} d(x,\Gamma)^r\mathrm{d}P(x) + \int_{C_c(\Gamma_1)} d(x,\Gamma)^r\mathrm{d}P(x)\\
&>  \sum_{b\in \Gamma} \int_{C_b(\Gamma_1)} d(x,\Gamma_1)^r\mathrm{d}P(x) + \int_{C_c(\Gamma_1)} d(x,\Gamma_1)^r\mathrm{d}P(x)\\
&= \int d(x, \Gamma_1)^r\mathrm{d}P(x) = G_{r,n}^P(a_1, \ldots, a_{i-1},c, a_{i+1}, \ldots, a_n).
\end{align*}
This yields a contradiction.  Consequently, the components $a_i$ of $a$ are pairwise distinct so that $\Gamma= \{a_1,\ldots, a_n\}$. In particular, we may assume that the neighbourhoods $U_i$ are pairwise disjoint.

\smallskip
Next assume that  $a_i \notin  {\rm argmin}_{y \in E} G_{r,1}^{P(\cdot|C)}(y)$ for some $i\!\in \{1, \ldots, n\}$ and some Borel set $C$ satisfying $W_{a_i}(\Gamma) \subset C \subset V_{a_i}(\Gamma)$, noting that $P(C) >0$. Then there exists $z\!\in E$ such that $G_{r,1}^Q(z) < G_{r,1}^Q(a_i)$ where $Q = P(\cdot\, |\, C)$.  Consider $y_s = s z +(1-s) a_i$, $s\!\in [0,1]$. Then $y_s\!\in U_i$ for $s>0$ small enough and using the convexity of $G_{r,1}^Q$ in case $r\ge 1$, we obtain
\[
G_{r,1}^Q(y_s) \le sG_{r,1}^Q(z)+(1-s)  G_{r,1}^Q(a_i) < G_{r,1}^Q(a_i).
\]
Consequently,
\[
\int_C \| x-y_s\|^r \mathrm{d}P(x) < \int_C \|x-a_i\|^r\mathrm{d}P(x).
\]
Hence, if $\{C_{a_j}(\Gamma): 1 \le j\le n\}$ denotes a Voronoi partition  of $E$ induced by $\Gamma$ with $C_{a_i}(\Gamma)= C$ and setting $\Gamma_2 = \big( \Gamma\setminus\{a_i\}\big) \cup \{y_s\}$ where $y_s \notin \Gamma$,
\begin{align*}
&G_{r,n}^P(a_1,\ldots, a_{i-1}, y_s, a_{i+1}, \ldots, a_n) = \int d(x, \Gamma_2)^r \mathrm{d}P(x)\\
& \qquad\qquad = \sum_{j=1} ^n\int_{C_{a_j}(\Gamma)} d(x,\Gamma_2)^r \mathrm{d}P(x) \\
&  \qquad\qquad \le  \sum_{j\neq i} \int_{C_{a_j}(\Gamma)} \|x-a_j\|^r \mathrm{d}P(x) + \int_{C} \|x-y_s\|^r \mathrm{d}P(x)\\
&  \qquad\qquad <  \sum_{j\neq i} \int_{C_{a_j}(\Gamma)} \|x-a_j\|^r \mathrm{d}P(x) + \int_{C} \|x-a_i\|^r \mathrm{d}P(x)\\
& \qquad\qquad = \int d(x,\Gamma)^r \mathrm{d}P(x) = G_{r,n}^P(a)
 \end{align*}
which again contradicts global minimality  of $G_{r,n}^P$ at $a$ on $U$.

\smallskip
\noindent $(b)$ In view of $(a)$, the proof of $(b)$ is literally the same as the proof of~\cite[Proposition~1.3.3]{LuPag23}.\hfill$\Box$
\end{proof}

\bigskip
  In general, local minima may neither be contained in ${\rm supp}(P)$ nor be pairwise distinct. Let e.g.\ $E=\R$, $P=U([-1,1])$, $n=2$, $a_1=0$ and $a_2 =3$. Then $(a_1,a_2)$ is a local minimizer of $G_{r,2}^P$, $r\!\in (0,+\infty)$.  If $n=3$ and $a_2=a_3$, then $(a_1, a_2,a_3)$ is a local minimizer of $G_{r,3}^P$. However, these local minima are not strict and have components  lying  outside the support of $P$.

\begin{thm}[Hilbert setting]\label{prop:1.3.5} Let $(E, \langle\cdot,\cdot\rangle)$ be a Hilbert space
and assume that ${\rm supp}(P)$ is convex. If $G_{r,n}^P$ has a strict local minimum at $a=(a_1,\ldots, a_n) \!\in E^n$, then $a_i \!\in {\rm supp}(P)$  for every $i\!\in \{1,\ldots,n\}$.  In particular, the conclusions of Theorem~\ref{thm:1.3.4} are true.

Moreover, if $r\ge 2$, $P(\{a_i\})=0$ for any $i\!\in \{1, \ldots,n\}$ when $r>2$ and ${\rm supp}(P)$ has nonempty interior, then $a_i$ is an interior point of ${\rm supp}(P)$ for every $i$.
\end{thm}

\begin{proof} Let $U = \prod_{i=1}^n U_i$ be a neighbourhood of $a$ such that $G_{r,n}^P$ has a strict (global) minimum at $a$ on $U$. Let $N$ denote the number of pairwise distinct components of $a$ so that there exist $i_1=1<i_1< \cdots<i_{_N} \le n$ such that $a_{i_1}, \ldots,a_{i_{_N}}$ are pairwise distinct. We may assume that the  neighborhoods $U_{i_j}$ are pairwise disjoint.  Set $\Gamma= \{a_{i_1}, \ldots, a_{_{i_N}}\}$ and $K = {\rm supp}(P)$. Assume $a_i \notin K$ for some $i\!\in \{i_1,\ldots,i_{_N}\}$. Let $b \!\in K$ be the (unique) $K$-valued best approximation of $a_i$ such that $\|a_i-b\| = d(a_i, K)$ or equivalently
\[
\langle a_i-b, x-b\rangle \le 0 \quad \mbox{ for every } x\!\in K.
\]
Consider $y_s = sb +(1-s) a_i$, $s\!\in [0,1]$. Then $y_s \!\in U_i$ for $s>0$ small enough and
\[
\|y_s-x\|^2 < \|a_i-x\|^2 \quad \mbox{ for every } x\!\in K.
\]

\noindent In fact
\begin{align*}
&\|y_s-x\|^2-\|a_i-x\|^2 = \| y_s-b +b-x\|^2 - \|a_i-b+b-x\|^2\\
& \qquad\qquad = \| y_s-b \|^2 +2\langle y_s-b, b-x\rangle - \|a_i-b\|^2-2 \langle a_i-b, b-x \rangle\\
& \qquad\qquad = \big( (1-s)^2-1\big) \|a_i-b\|^2 + \big( 2(1-s)-2\big)  \langle a_i-b, b-x\rangle <0.
\end{align*}
Set $\Gamma_2 =  \big(\Gamma\setminus\{a_i\}\big)\cup\{y_s\}$ where $y_s \notin \Gamma$. Let $\{C_c(\Gamma): c\!\in \Gamma\}$ denote a Voronoi partition induced by $\Gamma$. Then
\begin{align*}
&G_{r,n}^P (a_1,\ldots, a_{i-1}, y_s,a_{i+1},\ldots,a_n) = \int d(x, \Gamma_2)^r \mathrm{d}P(x)\\
&\qquad\qquad= \sum_{c \in \Gamma} \int_{C_c(\Gamma)}d(x, \Gamma_2)^r \mathrm{d}P(x)\\
&\qquad\qquad\le \sum_{c\in \Gamma, \, c\neq a_i}  \int_{C_c(\Gamma)} \| x-c\|^r \mathrm{d}P(x) +  \int_{C_{a_i}(\Gamma) \cap K}\| x-y_s\|^r \mathrm{d}P(x)\\
&\qquad\qquad\le \sum_{c\in \Gamma, \, c\neq a_i}  \int_{C_c(\Gamma)}\| x-c\|^r \mathrm{d}P(x) +  \int_{C_{a_i}(\Gamma) \cap K}\| x-a_i\|^r \mathrm{d}P(x)\\
&\qquad\qquad= \int d(x, \Gamma)^r \mathrm{d}P(x) = G_{r,n}^P(a).
\end{align*}
This yields a contradiction to strict minimality of $G_{r,n}^P$ at $a$ on $U$.

Now assume $r\ge 2$, $P(\{a_i\})= 0$ for every $i$ when $r>2$ and $\mathring{K} \neq \varnothing$. We already know that $a_i \!\in K$ for every $i$ so that by Theorem~\ref{thm:1.3.4} the components $a_i$ of $a$ are pairwise distinct and $\Gamma= \{a_1, \ldots, a_n\}$ is an $L^r$-stationary $n$-quantizer for $P$, noting that $K$ has no isolated points. Consequently, $\Gamma\subset \mathring{K} $ provided $P\big(\mathring{K}\cap W_{a_i}(\Gamma)\big)>0$ for every $i$ (see \cite[Remarks~1.3.4 and~1.3.5]{LuPag23}). This is  true since the open sets $\mathring{K} \cap W_{a_i}(\Gamma)$ are not empty. In fact fix $z\!\in \mathring{K}$ and let $y_s = s z +(1-s)a_i$. One checks that $y_s \!\in \mathring{K}$ for every $s\!\in (0,1]$ and $y_s \!\in W_{a_i}(\Gamma)$ for $s>0$ small enough.\hfill$\Box$
\end{proof}

\end{document}